%% file: Owen_Dick_Chen_arXiv.tex
\documentclass{article}
\usepackage{amsmath,amsfonts,amssymb,amsthm}
\usepackage{paralist}
\usepackage{esint}
\usepackage[round]{natbib}

\input{newcommandsnotbook}

\newcommand{\ul}{\underline}
\newcommand{\ol}{\overline}

\newcommand{\ult}{\underline\tau}

\newcommand{\ulfu}{{\underline f}_u}
\newcommand{\ulfv}{{\underline f}_v}

\newtheorem{theorem}{Theorem}
\newtheorem{proposition}{Proposition}
\newtheorem{lemma}{Lemma}

\newcommand{\sumdot}{\text{\tiny$\bullet$}}
\newcommand{\sd}{\text{\tiny$\bullet$}}

\newcommand{\bsj}{{\boldsymbol{j}}}
\newcommand{\ind}{{\mathbb{I}}}

\begin{document}

\title{Higher order Sobol' indices}
 \author{Art B. Owen\\ Stanford University
 \and
 Josef Dick\\University of New South Wales
 \and
 Su Chen\\ Two Sigma LLC
 }
\date{June 2013}

\maketitle
\begin{abstract}
Sobol' indices measure the dependence of a high dimensional function on groups of variables defined on the unit cube $[0,1]^d$. They are based on the ANOVA decomposition of functions, which is an $L^2$ decomposition. In this paper we discuss generalizations of Sobol' indices which yield $L^p$ measures of the dependence of $f$ on subsets of variables. Our interest is in values $p>2$ because then variable importance becomes more about reaching the extremes of $f$.
We introduce two methods. One based on higher order moments of the ANOVA
terms and another based on higher order norms of a spectral decomposition
of $f$, including Fourier and Haar variants. Both of our generalizations
have representations as integrals over $[0,1]^{kd}$ for $k\ge 1$,
allowing direct Monte Carlo or quasi-Monte Carlo estimation. We find
that they are sensitive to different aspects of $f$, and thus
quantify different notions of variable importance.
\end{abstract}

\section{Introduction}

Sobol' indices \citep{sobo:1990} are
the standard way to measure the importance
of variables and subsets of variables
for a black box function defined
on the unit cube $[0,1]^d$.
These measures are used in applications in
aerospace engineering and climate models
among many others.

Sobol's indices are based on the ANOVA decomposition
of $[0,1]^d$, which is an $L^2$ method.
An aeronautics-astronautics
engineering student, Gary Tang, asked us about how
to construct an alternative to Sobol'
indices that would identify which variables
are most important when one is especially interested
in the extreme values taken on by the function.
In this paper we address that problem
by considering alternative measures
based on other criteria
that place greater emphasis on extremes than
$L^2$ does.

Perhaps the simplest way to get an index more
sensitive to extremes in $f$ is to replace
the target function $f(\bsx)$ by a transformed version such as $|f(\bsx)|$ or
$\exp(f(\bsx))$ or $1_{f(\bsx)\ge M}$ for a threshold $M$
and so on, followed by an application of the
usual Sobol' indices. This approach will often be reasonable.
In some cases though, it may complicate the problem.
For example, if $f$ is a sum of functions of one
variable at a time, then $f^2$ involves pairwise interactions
that were not present in $f$, and
$1_{f(\bsx)\ge M}$ may involve interactions  of all orders.
Furthermore, if $f$ only takes two values, such as $0$
or $1$, (e.g., safe versus dangerous outcomes),
then transforming it to take two different
values does not help.  As a result, we consider new generalizations.

The ANOVA can be developed as an analysis of
$L^2[0,1]^d$, or as a synthesis of Fourier, Walsh or other
basis expansions. Both of these methods can be
used to make $L^p$ generalizations.
Additionally, the Sobol' indices satisfy some
identities that can be directly generalized.
These approaches coincide for $p=2$, but they differ for $p\ne 2$.

An outline of this paper is as follows.
Section~\ref{sec:nota} introduces our notation
and reviews the ANOVA and Sobol' indices.
Section~\ref{sec:anal} presents some related
non-$L_2$ concepts, median polish and analysis
of skewness, from the literature.
One of our methods includes a crossed-effects
extension of the analysis of skewness as a special case.
Section~\ref{sec:iden} presents a generalization
based on extending one of Sobol's identities to
$p$'th order moments. The identity yields a
representation of the index as an integral
of dimension $dp$ or lower.
For even integers $p\ge 2$ we show that the resulting
estimates are nonnegative and increase when any
set of variables is replaced by a superset.
Section~\ref{sec:synt}
presents a generalization based on the synthesis
from a Fourier expansion.
When $p\ge 2$ is an even integer, then
the resulting importance measures are sums of
$p$'th powers of the moduli of the function's Fourier coefficients.
Yet they can still be estimated directly by a high dimensional
quadrature, based on an identity like one of Sobol's.
That integral can be converted into one of dimension $d(p-1)$
or lower.
We also provide a version based on Walsh functions,
which again has nonnegativity and additivity
when $p\ge 2$ is an even integer and also has an
integral representation for quadrature.
For odd $p$, we include a `Dirichlet kernel trick' that
produces non-negative importance measures based on $L_p$ 
norms of Fourier or Walsh coefficients. That method
also allows one to favor certain parts of the spectrum.

Section~\ref{sec:spec} illustrates our importance
measures on test functions that are sums or products.
We use such examples to confirm
that our measures focus on variables that bring $f$
towards extreme values.
For product functions, and even $p$, our spectral measures find that
the most important variables are those
whose spectrum is sparsest. Our moment measure, for $p=4$,
favors variables with high kurtosis and with mean
and skewness of the same sign.
We look also at the important special case
a rectangular spike:
$f(\bsx)=\prod_{j=1}^d 1_{x_j\le \epsilon_j}$.
When $f$ measures hitting a small region like this
the variable with the smallest $\epsilon_j$ is the
most important one, at least when all $\epsilon_j$
are small. Both moment and Fourier
measures favor small $\epsilon_j$.
For additive functions, having no interactions, we find that the spectral
measures  place all their importance on singleton sets.
The moment measure does this for third but not fourth moments.
Section~\ref{sec:disc} has a discussion.

\section{Notation}\label{sec:nota}
We are given a real-valued function
$f$ defined on $[0,1]^d$ for $d\ge 1$
and we are interested in quantifying
the importance to $f$ of various subsets
of the variables in the set $\cd = \{1,2,\dots,d\}$.

We make frequent use of subsets of $\cd$
as indices. The complement of $u\subseteq\cd$
is $u^c=\cd-u$, or simply $-u$ when that is typographically more
convenient.  The cardinality
of $u$ is $|u|$.  For $\bsx\in[0,1]^d$,
the point $\bsx_u\in[0,1]^{|u|}$ is made up of $x_j$
for $j\in u$ and $\mrd \bsx_u=\prod_{j\in u}\mrd x_j$.
We use $u\subset v$ to mean
that $u$ is a proper subset of $v$ (i.e., $u\subsetneq v$).

We often make a new point from components of
two old points. If $\bsx,\bsz\in[0,1]^d$
and $u\subseteq \cd$, then
$\bsy\equiv \bsx_u\glu\bsz_{-u}$ is the point in $[0,1]^d$
with $y_j=x_j$ for $j\in u$ and $y_j=z_j$ for $j\not\in u$.

\subsection{ANOVA of $[0,1]^d$}

The ANOVA decomposition represents $f(\bsx)$ via
\begin{align}\label{eq:anova}
f(\bsx) = \sum_{u\subseteq\cd} f_u(\bsx)
\end{align}
where the functions $f_u$ are defined recursively
by
\begin{align}\label{eq:analy}
f_u(\bsx) = \int_{[0,1]^{d-|u|}}\Bigl(
f(\bsx)-\sum_{v\subset u}f_v(\bsx)\Bigr)\rd \bsx_{-u}.
\end{align}
From usual conventions,
$f_\emptyset(\bsx) = \mu\equiv\int_{[0,1]^d}f(\bsx)\rd\bsx$
for all $\bsx\in [0,1]^d$.
The function $f_u$ only depends on $x_j$ for $j\in u$.
For $f\in L^2[0,1]^d$, these functions satisfy
$\int_0^1 f_u(\bsx)\rd x_j=0$ when $j\in u$,
from which it follows that
$\int f_u(\bsx)f_v(\bsx)\rd\bsx=0$ for $u\ne v$
and that
\begin{align}\label{eq:anovaresult}
\sigma^2 = \sum_{u\subseteq\cd}\sigma^2_u
\end{align}
where $\sigma^2=\int (f(\bsx)-\mu)^2\rd\bsx$,
$\sigma^2_\emptyset=0$ and
$\sigma^2_u = \int f_u(\bsx)^2\rd\bsx$
for $u\ne\emptyset$.
The name ANOVA stands for analysis of variance, as
given by~\eqref{eq:anovaresult}.
This decomposition goes back to \cite{hoef:1948}.

\cite{sobo:1969} obtained the decomposition~\eqref{eq:anova}
by a different route, described next.
Let $\phi_k$ for $k\in\ind$ be a complete orthonormal basis for $L^2[0,1]$,
where $\ind$ is a countable index set containing
a $0$ element,
with $\phi_0(x)=1$, $\forall x\in[0,1]$.
We can form the tensor product basis
$\phi_{\bsk}(\bsx)=\prod_{\ell=1}^d\phi_{k_\ell}(x_\ell)$,
for $\bsk\in\ind^d$ and then
$f(\bsx) = \sum_{\bsk\in\ind^d} \beta_{\bsk}\phi_{\bsk}(\bsx)$
where $\beta_{\bsk} = \int f(\bsx) \overline{\phi_{\bsk}}(\bsx) \rd\bsx$.
Then, with $\bszero$ a vector of $d$ zeros,
and $\ind_*$ the nonzero members of $\ind$,
\begin{align}\label{eq:synth}
f_u(\bsx) =
\sum_{\bsk_u\in\ind_*^{\,|u|}}
\beta_{\bsk_u\glu\bszero_{-u}}
\phi_{\bsk_u\glu\bszero_{-u}}(\bsx)
\end{align}
recovers the functions defined at~\eqref{eq:analy},
and $\sigma^2_u = \sum_{\bsj_u\in\ind_*^{\,|u|}}
\beta_{\bsj_u\glu\bszero_{-u}}^2$.
\cite{sobo:1969} used Haar functions for his
`decomposition into summands
of different dimensions' given by~\eqref{eq:synth}.
Where Hoeffding has an analysis,
Sobol' has a synthesis of variance.

\subsection{Sobol' indices and identities}

The importance of variable $j\in\cd$
is due in part to $\sigma^2_{\{j\}}$, but also
due to $\sigma^2_u$ for other sets $u$
with $j\in u$.  More generally, we may
be interested in the importance of a subset
$u$ of the variables.

Sobol' introduced two measures of variable
subset importance, which we denote
\begin{align*}
\underline\tau_u^2 & = \sum_{v\subseteq u}\sigma_v^2,\quad\text{and}\quad
\overline\tau_u^2 = \sum_{v\cap u\ne\emptyset}\sigma_v^2.
\end{align*}
These satisfy $\ul\tau_u^2\le\ol\tau_u^2$
and $\ul\tau_u^2+\ol\tau_{-u}^2=\sigma^2$.
Sobol' usually normalized these quantities by
$\sigma^2$, yielding
global sensitivity indices $\underline\tau_u^2/\sigma^2$
and $\overline\tau_u^2/\sigma^2$.  We will
use the unnormalized versions.

It is an elementary consequence of the ANOVA
definitions that
\begin{align}\label{eq:forult}
\iint f(\bsx)f(\bsx_u\glu\bsz_{-u})\rd\bsx\rd\bsz
=\mu^2 +\ult_u^2
\end{align}
and
\begin{align}\label{eq:forolt}
\frac12\iint \bigl(f(\bsx)-f(\bsx_{-u}\glu\bsz_{u})\bigr)^2\rd\bsx\rd\bsz
=\ol\tau_u^2.
\end{align}
We write these integrals over
$(\bsx,\bsu)\in[0,1]^{2d}$, although the first
really only uses $2d-|u|$ components and
the second uses $d+|u|$.

The great convenience of Sobol's measures is
that they can be directly estimated by integration
without bias.
We do not need to explicitly estimate, square,
integrate and sum the individual ANOVA terms.
As a consequence, we can avoid numerical optimization
and bias corrections.

 It is computationally convenient to replace
 equation~\eqref{eq:forult} by
 \begin{align}\label{eq:forult2}
 \iint f(\bsx)\bigl(f(\bsx_u\glu\bsz_{-u})-f(\bsz)\bigr)\rd\bsx\rd\bsz
 =\ult_u^2,
 \end{align}
 because it eliminates the need to
 subtract an estimate of $\mu$.
 Equation~\eqref{eq:forult2} was developed
 independently in \cite{salt:2002} and 
 by \cite{maun:2002},
 and it performs better when $\ult^2_u$ is small.
 For discussion and another estimator, see
 \cite{newsobol:tr}.

\subsection{Generalizations}

We have three different ways to generalize
the ANOVA to higher moments.  First, we can
generalize the original ANOVA decomposition
by noticing that the integrals in it minimize
a quadratic quantity, and then replacing that
quadratic by a higher order moment.
Second, we can generalize the Sobol' indices
directly, replacing the integrals of products
of pairs of function values by integrals of
products of three or more function values.
Third, we can generalize Sobol's synthesis.

\section{Related literature}\label{sec:anal}

In this section we consider two non-$L_2$ methods
from the literature.  
A natural approach to generalizing the ANOVA
to $p\ne 2$ begins with
the probabilistic interpretation of $f_u(\bsx)$
as a conditional expectation
$$
f_u(\bsx) = \e\Bigl( f(\bsx)-\sum_{v\subset u}f_v(\bsx)\mid \bsx_u\Bigr).
$$
For any $\bsx_u\in[0,1]^d$
$$
f_u(\bsx)=
\argmin_m\e\biggl( \Bigl(f(\bsx)-\sum_{v\subset u}f_v(\bsx)-m\Bigr)^2
\mid \bsx_u\biggr).
$$

Just as the conditional expectation minimizes
conditional variance, we may generalize the ANOVA
to moments $p\ge 1$, via
$$
f_u^{(p)}(\bsx)=
\argmin_m\e\biggl( \Bigl|f(\bsx)-\sum_{v\subset u}f^{(p)}_v(\bsx)-m\Bigr|^p
\mid \bsx_u\biggr).
$$
This generalization satisfies
$f(\bsx)=\sum_uf^{(p)}_u(\bsx)$
through the definition
of $f_{\cd}^{(p)}$, but the terms in it are not
generally orthogonal.  Nor do they decompose
$\int |f(\bsx)|^p\rd\bsx$, nor do they
generally integrate to $0$ over $x_j$
for $j\in u$.
If $|f|$ is bounded, then
there is a $p=\infty$ version
corresponding to a statistic called the midrange. 

It is cumbersome to minimize norms other than
$L_2$ to define alternatives to $f_u$.
The one example we found for this approach is
the median polish method, in the next section.
It uses  $p=1$,
which might be expected to place {\sl less}
emphasis on extremes of $f$ than the ANOVA, 
and is based on conditional medians.

\subsection{Median polish}

\cite{tuke:1977}
describes the median polish algorithm
for a two dimensional table of numbers
$X_{ij}$, $i=1,\dots,I$ and $j=1,\dots,J$.
The median polish algorithm generates a decomposition
$$X_{ij} = a_i + b_j + R_{ij}.$$
Starting with $a_i=b_j=0$ and $R_{ij}=X_{ij}$,
it alternates between row steps
\begin{align*}
m_i &\gets\text{median}(R_{i1},\dots,R_{iJ}),\quad 1\le i\le I\\
a_i &\gets a_i+m_i,\quad 1\le i\le I\\
R_{ij} &\gets R_{ij}-m_i,\quad 1\le i\le I, 1\le j\le J
\end{align*}
and analogous column steps.
\cite{sieg:1983} shows that the algorithm converges
when all of the $X_{ij}$ are rational numbers.
While median polish will converge to a
result where every row and column of $R_{ij}$
has median $0$, the result does necessarily
have the $L_1$ minimizing values of $a_i$ and $b_j$.
For a table of data with an even number
$I=2k$ of rows, \cite{sieg:1983} gets better results via
the 'low median', which is the $k$'th smallest value,
instead of the median which averages the $k$'th
and $k+1$'st values.

In principal one could evaluate $f$ on
a grid embedded in $[0,1]^2$ and
apply the median polish algorithm.
While there may be reasonable ways to generalize
median polish to $d>2$, the necessity of
estimating the additive components in order
to measure them is computationally unattractive.

\subsection{Analysis of skewness}

\cite{wang:2001} defines an analysis of skewness
for problems in biology.
Let $X_{ij}$ be a measure on animal $j=1,\dots,n_i$
from population $i=1,\dots,I$.
Here animals are nested within populations and
the appropriate analysis of variance is:
$$
\sum_{i=1}^I\sum_{j=1}^{n_i}(X_{ij}-\bar X_{\sumdot\sumdot})^2
=\sum_{i=1}^In_i(\bar X_{i\sumdot}-\bar X_{\sumdot\sumdot})^2
+\sum_{i=1}^I\sum_{j=1}^{n_i}(X_{ij}-\bar X_{i\sumdot})^2.
$$

An analogous analysis of skewness is
\begin{align*}
\sum_{i=1}^I\sum_{j=1}^{n_i}(X_{ij}-\bar X_{\sumdot\sumdot})^3
&=\sum_{i=1}^In_i(\bar X_{i\sumdot}-\bar X_{\sumdot\sumdot})^3
+\sum_{i=1}^I\sum_{j=1}^{n_i}(X_{ij}-\bar X_{i\sumdot})^3\\
&+ 3\sum_{i=1}^I(\bar X_{i\sumdot}-\bar X_{\sumdot\sumdot})
\sum_{j=1}^{n_i}(X_{ij}-\bar X_{i\sumdot})^2.
\end{align*}
The terms above correspond to skewness of
group means, skewness of observations within
groups and a third term measuring the correlation
of within group variance and the group mean.
The relative sizes of these terms have been interpreted
in terms of driven versus passive trends in evolutionary
biology.
The total skewness can be negative as can any of
its terms.

The analysis is centered on $\bar X_{\sd\sd}$
which is not generally the minimizer of $\sum_i\sum_j|X_{ij}-m|^3$
over $m\in\real$. Similarly, $\bar X_{i\sd}$
minimizes $\sum_j|X_{ij}-m|^2$ not
 $\sum_j|X_{ij}-m|^3$.
In other words, this method is not based on generalizing
the successive minimization property of ANOVA terms.

For functions on the unit cube, we can develop
an analysis of skewness. A crossed decomposition
is more appropriate than a nested one.
Let $f(\bsx)=\mu+\sum_{u\ne\emptyset}f_u(\bsx)$
be the ANOVA decomposition of $f$.
Then
\begin{align*}
\int (f(\bsx)-\mu)^3\rd\bsx
& = \sum_{u\ne\emptyset}\sum_{v\ne\emptyset}\sum_{w\ne\emptyset}\int f_u(\bsx)f_v(\bsx)f_w(\bsx)\rd \bsx
\end{align*}
The product $f_uf_vf_w$ has mean zero if
there is some index $j$ that belongs to precisely
one of the sets $u$, $v$, $w$.
There can be more nonzero terms than nonempty
subsets of $\cd$. For example
$f_{\{1,2\}}(\bsx)f_{\{2,3\}}(\bsx)f_{\{1,3\}}(\bsx)$
need not integrate to zero.
After eliminating the terms that must be zero, we
find that $\int (f(\bsx)-\mu)^3\rd\bsx$ equals
\begin{align*}
\sum_{u\ne\emptyset}\int f_u(\bsx)^3\rd\bsx
&+\sum_{u\ne\emptyset}\sum_{v\ne u\atop v\ne\emptyset}
\sum_{z\subset u\cap v}
\int f_u(\bsx)f_v(\bsx)f_{(u\Delta v)\cup z}\rd\bsx.
\end{align*}

For example, with $d=2$,
there are $3$ nonempty subsets
of $\{1,2\}$ providing $27$ combinations
for $u$, $v$ and $w$ of which only $12$ vanish,
yielding
\begin{align*}
\int (f(\bsx)-\mu)^3\rd\bsx
& =
\int f_{\{1\}}^3(\bsx)\rd\bsx
+\int f_{\{2\}}^3(\bsx)\rd\bsx
+\int f_{\{1,2\}}^3(\bsx)\rd\bsx\\
& +
3\int f_{\{1\}}(\bsx)f_{\{1,2\}}^2(\bsx)\rd\bsx
+3\int f_{\{2\}}(\bsx)f_{\{1,2\}}^2(\bsx)\rd\bsx\\
&+6\int f_{\{1\}}(\bsx)f_{\{2\}}(\bsx)f_{\{1,2\}}(\bsx)\rd\bsx.
\end{align*}

Our moment based method in Section~\ref{sec:iden} 
provide crossed decompositions for $d$ dimensions
and $p$'th powers. The terms are sums together
into $2^d-1$ effects.

\section{Generalizing the Sobol' identity}\label{sec:iden}

Instead of generalizing the ANOVA to higher
moments, we find it more convenient to directly
generalize the identity~\eqref{eq:forult}
which yields $\mu^2+\ult^2_u$. We are
generalizing $\mu^2 + \ult^2_u$
instead of $\ult^2_u$, because
the minimizer of $\int |f(\bsx)-m|^p\rd\bsx$
over $m$, is the mean when $p=2$, but is otherwise
not easy to identify.

Where~\eqref{eq:forult} uses $2$ points
in $[0,1]^d$ with common $\bsx_u$,
our generalization works via
$p\ge2$ such points.
Define $\ult^{(p)}$ via
\begin{align}\label{eq:ult(p)}
\ul\tau^{(p)}_u+\mu^p=
\int\!\!\cdots\!\!\int
\prod_{k=1}^p f(\bsx_u\glu\bsz^{(k)}_{-u})
\rd\bsx \prod_{k=1}^p\rd\bsz^{(k)}
\end{align}
where $\bsz^{(1)},\dots,\bsz^{(p)}\in[0,1]^d$.
This integral is over $[0,1]^{(p+1)d}$ but only uses $|u|+ p(d-|u|)$
components.
For $p=2$, we get the usual Sobol' sensitivity
indices (plus $\mu^2$).
The desirable property of~\eqref{eq:ult(p)} is that it is a
multivariable integral and may be estimated by Monte Carlo
or quasi-Monte Carlo sampling without requiring any numerical
optimization.

When we seek to estimate $\ult^{(p)}_u$ it is necessary
to subtract an estimate of $\mu^p$.
One approach, generalizing an estimate studied in
\cite{jano:klei:lagn:node:prie:2012:tr}
is to use
\begin{align}\label{eq:genjano}
\wh\ult^{(p)}_u
&=\frac1n\sum_{i=1}^n
\prod_{k=1}^pf(\bsx_{i,u}\glu\bsz_{i,-u}^{(k)})
 -\hat\mu^p,\quad\text{where}\\
\hat\mu&=\frac1{np}\sum_{i=1}^n\sum_{k=1}^p
f(\bsx_{i,u}\glu\bsz_{i,-u}^{(k)}).
\end{align}
A second approach, generalizing an estimate in
\cite{maun:2002} and \cite{salt:2002} takes
\begin{align}\label{eq:genmaunsalt}
\wh\ult^{(p)}_u
&=\frac1n\sum_{i=1}^n
\Bigl(
\prod_{k=1}^pf(\bsx_{i,u}\glu\bsz_{i,-u}^{(k)})
 -\prod_{k=1}^pf(\bsz_i^{(k)})
\Bigr),
\end{align}
a sample version of the identity
$$
\ult^{(p)}_u =
\int\biggl(
\prod_{k=1}^pf\bigl(\bsx_{u}\glu\bsz_{-u}^{(k)}\bigr)
-\prod_{k=1}^pf\bigl(\bsz_{i}^{(k)}\bigr)
\biggr)\rd\bsx\prod_{k=1}^p\rd\bsz^{(k)}.
$$
Equation~\eqref{eq:genmaunsalt} provides unbiased estimates
of $\ult^{(p)}_u$.
Even for $p=2$ it is known that
neither estimate~\eqref{eq:genjano} or~\eqref{eq:genmaunsalt}
is always better than the other.
For instance \cite{sobolmatrix}
finds that~\eqref{eq:genmaunsalt}
is more accurate in some examples
with small $\ult^2_u$, while~\eqref{eq:genjano}
is better on some examples with large $\ult^2_u$.

The most interesting cases are $p=3$, which gives us
a skewness measure for each subset of variables,
and $p=4$, the smallest even power above $2$.
For even integers $p\ge 4$ we
get nonnegative measures that are increasing in~$u$
as shown below.
We will use
\begin{align}\label{eq:ulf}
\ulfu(\bsx) = \sum_{v\subseteq u}f_v(\bsx) = \e\bigl( f(\bsx)\mid\bsx_u\bigr),
\end{align}
when $\bsx\sim\dustd[0,1]^d$.

\begin{proposition}\label{prop:ascondmu}
For integer $p\ge 1$,
$\ul\tau_u^{(p)} + \mu^p = \e\bigl( \ulfu(\bsx_u)^p\bigr)$.
\end{proposition}
\begin{proof}
Define $h(\bsx) = f(\bsx)-\ulfu(\bsx_u)$
and $\bsy_k = \bsx_u\glu\bsz_{-u}^{(k)}$, for $k=1,\dots,p$.
Then $\e(h(\bsy_j)\mid\bsx_u)=0$ and
\begin{align*}
\mu^p+\ul\tau_u^{(p)}
& = \e\biggl( \e\biggl(\,\prod_{k=1}^p (\ulfu(\bsx_u)+h(\bsy_k))
\mid\bsx_u\biggr)\biggr)\\
& = \e\bigl( \e(\ulfu(\bsx_u)^p\mid\bsx_u)\bigr) \\
& = \e\bigl( \ulfu(\bsx_u)^p\bigr).\qedhere
\end{align*}
\end{proof}


\begin{theorem}\label{thm:monotone}
Let $f\in L^p[0,1]^d$ 
for an even integer $p\ge 2$.
Then
$
\ul\tau_u^{(p)} \le \ul\tau_v^{(p)}
$
holds when $u\subseteq v\subseteq\cd$.
\end{theorem}
\begin{proof}
It suffices to consider the case
where $v=u\cup\{j\}$ for $j\not\in u$.
Let $h(\bsx) = \ulfv(\bsx_v)
-\ulfu(\bsx_u)
=\sum_{w\subseteq u}f_{w\cup \{j\}}(\bsx)$.
Then by Proposition~\ref{prop:ascondmu},
\begin{align*}
\mu^p+\ul\tau_v^{(p)}
&=\e( \ulfv (\bsx_v)^p)
=\e( (\ulfu (\bsx_u)+h(\bsx))^p)\\
&=\e( \e( (\ulfu (\bsx_u)+h(\bsx))^p\mid\bsx_u))\\
&\ge\e( \e( \ulfu (\bsx_u)^p\mid\bsx_u))
=\mu^p+\ul\tau_u^{(p)}
\end{align*}
by convexity of the function $\varphi(y)=y^{p}$.
\end{proof}

From Theorem~\ref{thm:monotone}, we see
that $\ul\tau^{(p)}_u$ has some important
properties for a subset importance quantity
when $p$ is an even integer.
First
$\ul\tau_u^{(p)}\ge\ul\tau_\emptyset^{(p)}=0$,
and so the importance of every subset is nonnegative.
Second, increasing the number of components in
a subset does not make the measure smaller.
Both of these properties also hold for the
measure
$\ul\tau_u^{\varphi} = \e( \varphi(\ulfu(\bsx)))-\varphi(\mu)$
for convex non-negative functions $\varphi$,
but when $\varphi(y)$ is even
power of $y$, we have a convenient estimation
formula based on~\eqref{eq:ult(p)} that lets us
avoid having to compute an estimate of $\ulfu$.

Odd power variable measures like $\ul\tau_u^{(3)}$
do not have the nesting property of Theorem~\ref{thm:monotone}
and they can take negative values. Such negative values may be informative
and interpretable.  For example if $\ul\tau_{\{1\}}^{(3)}<0$
while $\ul\tau_{\{2\}}^{(3)}>0$
this may indicate that controlling $x_1$ is
more important for attaining (or avoiding) very
small values of $f$ while $x_2$ is more
important for large values of $f$.

\section{Generalizing the synthesis}\label{sec:synt}

In this section we introduce a multilinear
operator that allows a generalization of the
synthesis approach to ANOVA.
We use two different bases, Fourier and Walsh.

\subsection{Fourier synthesis}
For $0 \le j < p$ let $f_j:[0,1]^d \to \real$ have a Fourier
expansion
\begin{equation*}
f_j(\bsx) = \sum_{\bsk \in \mathbb{Z}^d} \widehat{f}_j(\bsk)
\mathrm{e}^{2\pi \mathrm{i} \bsk \cdot \bsx}.
\end{equation*}
For any $j\in\{0,1,\dots,p-1\}$ let its successor
be $j+ \equiv j+1\hspace{-1mm}\pmod{p}$
and its predecessor be $j- \equiv j-1\hspace{-1mm}\pmod{p}$.
Our multilinear operator is
\begin{eqnarray}\label{eq:multifourier}
\langle f_0,\ldots, f_{p-1} \rangle_p = \int_{[0,1]^{dp}}
\prod_{j=0}^{p-1} f_j\bigl(\{(-1)^j (\bsx_j-\bsx_{j+})\}\bigr)
\rd \bsx_0 \cdots \,\mathrm{d}\bsx_{p-1},
\end{eqnarray}
where $\{\bsz\} = \bsz -  \lfloor \bsz \rfloor$ is the fractional part
of $\bsz$ (componentwise).

The following result is the fundamental lemma,
giving a multilinear orthogonality property of the
operator on Fourier functions.
\begin{lemma}\label{lem_fundamental}
Let $p \ge 2$ be an integer and $\bsk_0,\ldots, \bsk_{p-1} \in \mathbb{Z}^d$ and let
$\phi_{\bsk}(\bsx) = \mathrm{e}^{2\pi \mathrm{i} \bsk \cdot \bsx}$.
Then
\begin{equation*}
\langle \phi_{\bsk_0},\ldots, \phi_{\bsk_{p-1}} \rangle_p =
\begin{cases}
1, & \bsk_j=(-1)^j\bsk_0,\quad j=1,\dots,p-1 \\
0, & \text{otherwise.}
\end{cases}
\end{equation*}
\end{lemma}

\begin{proof}
For $p$ even we have
\begin{align*}
\langle \phi_{\bsk_0},\ldots, \phi_{\bsk_{p-1}} \rangle_p
& =
\int_{[0,1]^{dp}} \mathrm{e}^{2\pi \mathrm{i} \sum_{j=0}^{p-1}
\bsk_j \cdot \{(-1)^j (\bsx_j- \bsx_{j+} ) \}}
\,\mathrm{d}\bsx_0 \cdots \,\mathrm{d}\bsx_{p-1} \\
& =
\prod_{j=0}^{p-1} \int_{[0,1]^{d}} \mathrm{e}^{2\pi
\mathrm{i} (-1)^j (\bsk_j + \bsk_{j-})\cdot\bsx_j}
\,\mathrm{d}\bsx_j
\end{align*}
and for $p$ odd we have
\begin{align*}
\langle \phi_{\bsk_0},\ldots, \phi_{\bsk_{p-1}} \rangle_p
& = \prod_{j=1}^{p-1} \int_{[0,1]^{d}} \mathrm{e}^{2\pi
\mathrm{i} (-1)^j (\bsk_j + \bsk_{j-})\cdot\bsx_j}
\,\mathrm{d}\bsx_j \int_{[0,1]^d} \mathrm{e}^{2\pi \mathrm{i} (\bsk_0 - \bsk_{p-1}) \bsx_0} \,\mathrm{d} \bsx_0.
\end{align*}
The integrals are $1$ if $\bsk_j= (-1)^j \bsk_0$ for $0 \le  j < p$
and $0$ otherwise, which implies the result.
\end{proof}

The function $\langle \cdot\,,\ldots,\cdot \rangle_p$ is symmetric and
multi-linear. For integers $p\ge 2$ we will use
\begin{align*}
\sigma_p(f) &
\equiv \langle f,\ldots, f \rangle_p \\
&=\sum_{\bsk_0,\ldots,\bsk_{p-1} \in \mathbb{Z}^d} \prod_{j=0}^{p-1}
\widehat{f}(\bsk_j) \int_{[0,1]^{dp}} \mathrm{e}^{2\pi \mathrm{i}
\sum_{j=0}^{p-1} (-1)^j \bsk_j \cdot (\bsx_j- \bsx_{j+})
} \,\mathrm{d}\bsx_0 \dots \,\mathrm{d}\bsx_{p-1}\\
& = \sum_{\bsk\in\mathbb{Z}^d} \widehat{f}(\bsk)^{\lceil p/2 \rceil} \widehat{f}(-\bsk)^{\lfloor p/2\rfloor}.
\end{align*}
If $f$ is a real-valued function we have $\widehat{f}(-\bsk) = \overline{\widehat{f}(\bsk)}$. If $p$ is an even integer we therefore get
\begin{equation*}
\langle f,\ldots,f\rangle_p = \sum_{\bsk\in\mathbb{Z}^d} \bigl|\widehat{f}(\bsk)\bigr|^p.
\end{equation*}


The ANOVA decomposition
$f(\bsx) = \sum_{u \subseteq \cd} f_u(\bsx)$,
has terms
\begin{equation*}
f_u(\bsx) = \sum_{\bsk_u \in \ints_*^{|u|}}
\widehat{f}(\bsk_u\glu\bszero_{-u}) \mathrm{e}^{2\pi \mathrm{i}
\bsk_u\cdot \bsx_u}.
\end{equation*}
The diagonality of
the multilinear operator~\eqref{eq:multifourier}
yields a $p$-fold orthogonality for the ANOVA terms:
\begin{lemma}
Let $f$ be as above and let $f=\sum_u f_u$ be the ANOVA
decomposition of $f$. Then for all $u_0,\ldots, u_{p-1} \subseteq
\cd$, such that there are $i,j \in \{0,\ldots, p-1\}$
with $u_i \neq u_j$, we have
\begin{equation*}
\langle f_{u_0}, \ldots, f_{u_{p-1}} \rangle_p = 0.
\end{equation*}
\end{lemma}
\begin{proof}
If $u_i \neq u_j$, then $\bsk_{u_j}\glu\bszero_{-u_j} \neq
- \bsk_{u_i}\glu\bszero_{-u_i}$ for all $\bsk_{u_j} \in
\ints_*^{|u_j|}$ and $\bsk_{u_i} \in \ints_*^{|u_i|}$. Then
\begin{equation*}
\int_{[0,1]^d} \mathrm{e}^{2\pi \mathrm{i}
(\bsk_{u_j}\glu\bszero_{-u_j} + \bsk_{u_i}\glu\bszero_{-u_i}) \cdot
\bsx} \,\mathrm{d}\bsx = 0.
\end{equation*}
The result follows now from Lemma~\ref{lem_fundamental}.
\end{proof}


\begin{lemma}\label{lem_fs}
Let $f$ be as above and let $f=\sum_u f_u$ be the ANOVA
decomposition of $f$. Then we have
\begin{equation*}
\sigma_p(f) = \sum_u \sigma_p(f_u).
\end{equation*}
\end{lemma}

\begin{proof}
Recall that $\langle f_{u_0},\ldots, f_{u_{p-1}} \rangle_p = 0$ unless
$u_0 = \cdots = u_{p-1}$.
Therefore expanding
$\sigma_p(f)  = \langle f, \cdots, f\rangle_p$ yields
\begin{align*}
\sum_{u_0,\ldots u_{p-1} \subseteq \cd} \langle f_{u_0},
\ldots, f_{u_{p-1}} \rangle_p
& =  \sum_{u \subseteq
\cd} \langle f_u,\ldots, f_u \rangle_p
 =  \sum_{u \subseteq\cd }
\sigma_p(f_u).\qedhere
\end{align*}
\end{proof}


The aim is to estimate $\sigma_p(f_u)$ or sums of those. We
investigate this in the following.
For $u\subseteq\cd$, define $\ult^{[p]}_u$ via
\begin{align}\label{eq:defultsyn}
\ul\tau^{[p]}_u+\mu^p=
\int_{[0,1]^{dp}}  \prod_{j=0}^{p-1}
f\bigl(\{(-1)^j (\bsx_{u,j}-\bsx_{u,j+})\}\glu\bsy_{-u,j}\bigr) &
\prod_{j=0}^{p-1}\mrd\bsx_{u,j}\prod_{j=0}^{p-1}\mrd\bsy_{u,j}.
\end{align}
Here $\ult^{[p]}_\emptyset=0$.

\begin{theorem}\label{thm:decomp}
Let $f\in L^p[0,1]^d$, for integer $p\ge 2$,
with ANOVA decomposition $f= \sum_u
f_u$. Then for any $u \subseteq \cd$ we have
$$
\ul\tau^{[p]}_u +\mu^p= \sum_{v \subseteq u} \sigma_p(f_v).
$$
\end{theorem}
\begin{proof}
Using the Fourier series representation of $f$ and
Lemma~\ref{lem_fundamental} we obtain
\begin{equation*}
\ul\tau^{[p]}_u +\mu^p
=  \sum_{\bsk_u \in \mathbb{Z}^{|u|}}
|\widehat{f}(\bsk_u\glu\bszero_{-u})|^{p} = \sum_{v \subseteq u}
\sigma_p(f_v).\qedhere
\end{equation*}
\end{proof}

Theorem~\ref{thm:decomp} shows that the
importance measures $\ult^{[p]}_u$ are sums
of contributions $\sigma_p(f_v)$ from $v\subseteq u$.
This generalizes a property of the ANOVA to
$p\ge 2$.

Theorem~\ref{thm:decomp} can be generalized in the following way. Let
$f_0,\ldots, f_{p-1}$ be functions in $L^p[0,1]^d$
for integer $p\ge 2$ with Fourier coefficients $\wh f_j(\bsk)$,
and $\mu_j  = \int f_j(\bsx)\rd\bsx$.
Next we set
\begin{align*}
& \ul\tau^{[p]}_u(f_0,\ldots, f_{p-1})  +\prod_{j=0}^{p-1}\mu_j \\
\equiv & \int_{[0,1]^{dp}}
\prod_{j=0}^{p-1} f_j\bigl((-1)^j \{\bsx_{u,j}-
\bsx_{u,j+}\}\glu\bsy_{-u,j}\bigr)
\prod_{j=0}^{p-1}\mrd\bsx_{u,j}\prod_{j=0}^{p-1}\mrd\bsy_{-u,j}.
\end{align*}
Then
$$
\ul\tau^{[p]}_u(f_0,\ldots, f_{p-1})
+\prod_{j=0}^{p-1}\mu_j
= \sum_{\bsk_u \in
\mathbb{Z}^{|u|}} \prod_{j=0}^{p-1} \widehat{f}_j((-1)^j \bsk_u\glu\bszero_{-u}).
$$

\subsection{Walsh synthesis}

Here we replace the Fourier functions
by the Walsh functions in an integer
base $b\ge 2$.
For $b=2$, the coefficients of Walsh functions are real values.
The index set is $\ind=\natu_0$
and then $\ind_*=\natu$.

For a non-negative integer $k$ with base $b$ representation
\[
   k = \kappa_{a-1} b^{a-1} + \cdots + \kappa_1 b + \kappa_0,
\]
with $\kappa_i \in \{0,1, \ldots, b-1\}$, we define the Walsh function
$\wal_{k}:[0,1) \rightarrow \{z \in \mathbb{C}: |z| = 1\}$ by
\[
  \wal_{k}(x) := \mathrm{e}^{2\pi \mathrm{i} (x_1 \kappa_0 + \cdots + x_a \kappa_{a-1})/b },
\]
for $x \in [0,1)$ with base $b$ representation $x =
x_1 b^{-1} + x_2 b^{-2} +\cdots $ (unique in the sense that
infinitely many of the $x_i$ must be different from $b-1$).

For dimension $s \ge 2$, $\bsx = (x_1, \ldots, x_s) \in [0,1)^s$ and $\bsk = (k_1,
\ldots, k_s) \in \mathbb{N}_0^s$ we define $\wal_{\bsk} : [0,1)^s
\rightarrow \{z \in \mathbb{C}: |z|=1\}$ by
\[
   \wal_{\bsk}(\bsx) := \prod_{j=1}^s \wal_{k_j}(x_j).
\]
For more information on Walsh functions see \cite{chrest:1955, fine:1949, walsh:1923}.

Let $f_j:[0,1]^d \to \mathbb{R}$ have a Walsh series expansion of
the form
\begin{equation*}
f_j(\bsx)= \sum_{\bsk \in \mathbb{N}_0^d} \widehat{f}_{\wal}(\bsk)
\wal_{\bsk}(\bsx).
\end{equation*}

For $x,y \in \{z \in \mathbb{R}: z \ge 0\}$ with base $b$ expansion $x = \sum_{i=w}^{-\infty} x_i b^{i}$ and $y = \sum_{i=w}^{-\infty} y_i b^i$ (unique in the sense that infinitely many of the $x_i$ and $y_i$ must be different from $b-1$) we set
$x \ominus y = \sum_{i=w}^{-\infty} z_i b^{i}$ where $z_i = x_i -
y_i \pmod{b}$ and $z_i \in \{0,\ldots, b-1\}$. For vectors $\bsx$ and $\bsy$ we define the operation $\ominus$ componentwise. Similarly we set $\bsx \oplus \bsy$ where we change the definition of $z_i$ to $z_i = x_i + y_i \pmod{b}$. We define $\ominus \bsx = \bszero \ominus \bsx$ and $( \ominus 1)^j \bsx = \bsx$ if $j$ is even and $\ominus \bsx$ otherwise.

We now define
\begin{equation*}
\langle f_0,\ldots, f_{p-1} \rangle_{p,\wal} = \int_{[0,1]^{dp}}
\prod_{j=0}^{p-1} f_j\bigl((\ominus 1)^j (\bsx_j \ominus \bsx_{j+})\bigr)
\,\mathrm{d} \bsx_0 \cdots \,\mathrm{d} \bsx_{p-1}.
\end{equation*}
With this definition we also have the fundamental lemma for the
Walsh system.

\begin{lemma}\label{lem_fundamental_wal}
Let $\bsk_0,\ldots, \bsk_{p-1} \in \mathbb{N}_0^d$. Then
\begin{equation*}
\langle \wal_{\bsk_0},\ldots, \wal_{\bsk_{p-1}} \rangle_{p,\wal} =
\left\{\begin{array}{ll} 1, & \bsk_j= \ominus^j \bsk_0,\quad j=1,\dots,p-1
\\ 0, & \mbox{otherwise}.
\end{array} \right.
\end{equation*}
\end{lemma}


All the remaining results and definitions can therefore be obtained
in an analogous manner. In particular for functions $f$ with ANOVA
decomposition $f = \sum_u f_u$ we have
\begin{equation*}
\sigma_{p,\wal}(f) = \sum_{u \subseteq \cd}
\sigma_{p,\wal}(f_u),
\end{equation*}
where
\begin{equation*}
\sigma_{p,\wal}(f) \equiv \langle f, \ldots, f \rangle_{p,\wal} =
\sum_{\bsk \in \mathbb{N}_0^d} [\widehat{f}_{\wal}(\bsk)]^{\lceil p/2 \rceil} [\widehat{f}_{\wal}(\ominus \bsk)]^{\lfloor p/2 \rfloor}.
\end{equation*}
If $p$ is even and $f$ a real-valued function, then we get
\begin{equation*}
\sigma_{p,\wal}(f) \equiv \langle f, \ldots, f \rangle_{p,\wal} =
\sum_{\bsk \in \mathbb{N}_0^d} \left|\widehat{f}_{\wal}(\bsk)\right|^{p}.
\end{equation*}

\begin{lemma}\label{lem:walshdecomp}
Let $f$ be as above and let $f=\sum_u f_u$ be the ANOVA
decomposition of $f$. Then we have
\begin{equation*}
\sigma_{p, \wal}(f) = \sum_u \sigma_{p, \wal}(f_u).
\end{equation*}
\end{lemma}
The proof of
Lemma~\ref{lem:walshdecomp} follows by the same arguments as the proof of Lemma~\ref{lem_fs}.

We may estimate $\sigma_{p, \wal}(f_u)$ or their sums in
the same way we did for their Fourier analogues $\sigma_{p}(f_u)$.
For $u\subseteq\cd$, define $\ult^{[p]}_{u,\wal}$ via
\begin{align*}
\ul\tau^{[p]}_{u, \wal} + \mu^p_{\wal} =
\int_{[0,1]^{dp}}  \prod_{j=0}^{p-1}
f\bigl(\{(\ominus 1)^j (\bsx_{u,j} \ominus \bsx_{u,j+})\}\glu\bsy_{-u,j}\bigr) &
\prod_{j=0}^{p-1}\mrd\bsx_{u,j}\prod_{j=0}^{p-1}\mrd\bsy_{u,j}.
\end{align*}
Here $\ult^{[p]}_{\emptyset, \wal} = 0$.

\begin{theorem}
Let $f\in L^p[0,1]^d$, for integer $p\ge 2$,
with ANOVA decomposition $f= \sum_u
f_u$. Then for any $u \subseteq \cd$ we have
$$
\ul\tau^{[p]}_{u, \wal} +\mu^p_{\wal} = \sum_{v \subseteq u} \sigma_{p, \wal}(f_v).
$$
\end{theorem}
The proof of this result follows along the same lines as the proof of Theorem~\ref{thm:decomp}.

In general, for $p>2$,
$\sigma_p(f)$ and $\sigma_{p,\wal}(f)$ are different
and the Walsh measure will depend on
the base $b$ that was used. Parseval's
identity implies that $\sigma_2(f) = \sigma_{2,\wal}(f)$.

\subsection{Change of variable and dimension reduction}

Our $p$-fold inner products are defined
through a $pd$ dimensional integral.
But they are equivalent to a $(p-1)d$ dimensional
integral.

\begin{lemma}\label{lem:dimreduce}
For integers $p\ge 2$ and $d\ge 1$,
let $f_0,f_1,\dots,f_{p-1}\in L^p[0,1)^d$.
Then
\begin{align*}
&\int_{[0,1]^{dp}}\prod_{j=0}^{p-1}
f_j(\{(-1)^j (\bsx_j-\bsx_{j+})\})\prod_{j=0}^{p-1}\rd\bsx_j\\
& =
\int_{[0,1]^{dp-d}}
f_0(\bsy_0)f_1(\bsy_1)\cdots f_{p-2}(\bsy_{p-2})
f_{p-1}\bigl(\{\bsy_0-\bsy_1+ \dots + (-1)^{p-2} \bsy_{p-2}\}\bigr)
\prod_{j=0}^{p-2}\rd\bsy_j.
\end{align*}
\end{lemma}

\begin{proof}
We prove it for $p=4$; the general case uses
the same argument.
For $\bsx_0,\dots,\bsx_{3}\in[0,1)^d$
let $\bsy_0,\dots,\bsy_{3}$ be defined by
$$
\begin{pmatrix}
\bsy_{0}\\
\bsy_{1}\\
\bsy_{2}\\
\bsy_{3}
\end{pmatrix}
=
\begin{pmatrix}
1 & -1 & \phm0 & \phm0\\
0 & -1 & \phm1 & \phm0\\
0 & \phm0 & \phm1 & -1 \\
0 & \phm0 & \phm0 & -1
\end{pmatrix}
\begin{pmatrix}
\bsx_{0}\\
\bsx_{1}\\
\bsx_{2}\\
\bsx_{3}\\
\end{pmatrix}\tmod 1
$$
where both the matrix multiplication and the
modulus are taken componentwise.
This transformation
has Jacobian $1$ almost everywhere.
To simplify the integrals, we extend each $f_j$ to
a periodic function on $\real^d$, allowing
us to remove the $\{\cdots\}$ operation.
Making the change of variable,
\begin{align*}
&\iiiint f_0(\bsx_0-\bsx_1)f_1(\bsx_2-\bsx_1)f_2(\bsx_2-\bsx_3)f_3(\bsx_0-\bsx_3)
\rd \bsx_0\rd \bsx_1\rd \bsx_2\rd \bsx_3\\
 = &
\iiiint f_0(\bsy_0)f_1(\bsy_1)f_2(\bsy_2)f_3(\bsy_0-\bsy_1+\bsy_2)
\rd \bsx_0\rd \bsx_1\rd \bsx_2\rd \bsy_3\\
 = &
\iiint f_0(\bsy_0)f_1(\bsy_1)f_2(\bsy_2)f_3(\bsy_0-\bsy_1+\bsy_2)
\rd \bsy_0\rd \bsy_1\rd \bsy_2.\qedhere
\end{align*}
\end{proof}

Lemma~\ref{lem:dimreduce} also applies for the Walsh case. We have
\begin{align*}
&\int_{[0,1]^{dp}}\prod_{j=0}^{p-1}
f_j(\{(\ominus 1)^j (\bsx_j \ominus \bsx_{j+})\})\prod_{j=0}^{p-1}\rd\bsx_j\\
& =
\int_{[0,1]^{dp-d}}
f_0(\bsy_0)f_1(\bsy_1)\cdots f_{p-2}(\bsy_{p-2})
f_{p-1}\bigl(\{\bsy_0\ominus \bsy_1 \oplus \dots \oplus (\ominus 1)^{p-1} \bsy_{p-1}\}\bigr)
\prod_{j=0}^{p-2}\rd\bsy_j.
\end{align*}

\subsection{Weighted coefficients}\label{subsec_weighted_coeff}

The quantity $\langle f,f,\dots,f,g\rangle_{p+1}$
is also of interest for special choices of the
function $g$.
The result is to give weighted sums
of powers of the Fourier (or Walsh) coefficients.
We take $p$ to be an odd integer and
$g$ to be a weighting function.

Of particular interest is the Dirichlet kernel
\begin{equation*}
D_N(\bsx) = \sum_{\bsk \in \{-N,\ldots, N\}^d} \mathrm{e}^{2\pi \mathrm{i} \bsk \cdot \bsx} = \prod_{j=1}^d \frac{\sin (2\pi (N+1/2) x_j)}{\sin (\pi x_j)}.
\end{equation*}
If $x_j = 0$ or $1$ we set ${\sin (2\pi (N+1/2) x_j)}/{\sin (\pi x_j)} := 2N+1$.
For an odd integer $p > 1$ we have
\begin{equation*}
\langle f, \ldots, f, D_N \rangle_{p} = \sum_{\bsk \in \{-N,\ldots, N\}^d} \bigl|\widehat{f}(\bsk) \bigr|^{p-1}.
\end{equation*}
The result is a non-negative  importance measure for $f$ apart
from its very highest spatial frequencies.
Further, for $\boldsymbol{m} \in \mathbb{Z}^d$ we have
\begin{equation*}
\langle f, \ldots, f, D_N \mathrm{e}^{2\pi \mathrm{i} \boldsymbol{m} \cdot \bsx} \rangle_{p} = \sum_{\bsk \in \{-N,\ldots, N\}^d} \bigl|\widehat{f}(\bsk + \boldsymbol{m}) \bigr|^{p-1}.
\end{equation*}

The Dirichlet kernel for the Walsh system in base $b$ is
\begin{equation*}
D_{m, \wal}(\bsx) = \sum_{\bsk \in \{0,\ldots, b^m-1\}^d} \wal_{\bsk}(\bsx) = \prod_{j=1}^d 1_{[0,b^{-m})}(x_j).
\end{equation*}
Thus for odd integer $p > 1$ we have
\begin{equation*}
\langle f, \ldots, f, D_{m,\wal} \rangle_{p, \wal} = \sum_{\bsk \in \{0,\ldots, b^m-1\}^d} \bigl|\widehat{f}_{\wal}(\bsk) \bigr|^{p-1}.
\end{equation*}
Further, for $\boldsymbol{a} \in \mathbb{N}_0^d$ we have
\begin{equation*}
\langle f, \ldots, f, D_{m,\wal} \wal_{\bsa} \rangle_{p, \wal} = \sum_{\bsk \in \{0,\ldots, b^m-1\}^d} \bigl|\widehat{f}_{\wal}(\bsk \oplus \bsa) \bigr|^{p-1}.
\end{equation*}

\section{Special case functions}\label{sec:spec}

Here we consider some simple functional
forms for which our analysis can be carried
out in closed form.
The first ones are functions of product form,
including rectangular spikes.
We will see the effects of third and fourth
moments on the $\ult^{(p)}_u$ and the effects
of spectral sparsity on $\ult^{[p]}_u$.
The second are additive functions where we will
see the spectral method does not introduce any
apparent interactions.

The original Sobol' indices relate to variance
components via a Moebius relation
$$
\sigma^{2}_u = \sum_{v\subseteq u}(-1)^{|u-v|}\ult^{2}_v,
$$
for $u\ne\emptyset$.
Recalling that $\ult^{(p)}_u$, $\ult^{[p]}_u$ and $\ult^{[p]}_{u, \wal}$
are generalizations of $\mu^2+\ult^2_u$,
we can define analogues of variance components via
\begin{align}
\sigma^{(p)}_u &= \sum_{v\subseteq u}(-1)^{|u-v|}\ult^{(p)}_v,
\label{eq:newsigsqra}\\
\sigma^{[p]}_u &= \sum_{v\subseteq u}(-1)^{|u-v|}\ult^{[p]}_v,\quad\text{and} \label{eq:newsigsqrb} \\ \sigma^{[p]}_{u, \wal} & = \sum_{v \subseteq u} (-1)^{|u-v|} \ult^{[p]}_{v, \wal},
\end{align}
for $u\ne\emptyset$.
We also have $\sigma^{(p)}_\emptyset
=\sigma^{[p]}_\emptyset = \sigma^{[p]}_{\emptyset, \wal} =0$.

\subsection{Product functions}

Product functions are frequently used as examples
for sensitivity measures. A notable example is \cite{sobo:1993}.
Throughout this subsection we suppose that
\begin{align}\label{eq:defprod}
f(\bsx) = \prod_{j=1}^d h_j(x_j)
\equiv \prod_{j=1}^d (\mu_j+\tau_jg_j(x_j))
\end{align}
for real-valued functions $g_j$ and $h_j$ defined on $[0,1]$.
The functions $g_j$ satisfy $\int_0^1g_j(x)\rd x=0$
and $\int_0^1g_j(x)^2\rd x=1$.

The ANOVA components of a product function are
 $\sigma^2_u = \prod_{j\in u}\tau^2_j\prod_{j\not\in u}\mu_j^2$
for $u\ne\emptyset$. For a product function
$\mu^2+\ult^2_u = \prod_{j\in u}(\mu_j^2+\tau^2_j)\prod_{j\not\in u}\mu_j^2$.
An important subset of variables must include any $j$ with $\mu_j=0$.
When $\mu\ne 0$ we may write
$$
\mu^2+\ult^2_u = \mu^2\prod_{j\in u}(1+\tau^2_j/\mu_j^2)
$$
and then see that coefficients of variation $\upsilon_j=\tau_j/\mu_j$
govern importance.

We need $\int_0^1|f(x)|^p\rd x<\infty$ to make the
importance measures finite.
We will use
$\gamma_j = \int_0^1g_j^3(x)\rd x$
and
$\kappa_j = \int_0^1g_j^4(x)\rd x$
which we assume are finite.
If $x\sim\dustd(0,1)$, then $\gamma_j$ is
the skewness of $g_j(x)$ and
$\kappa_j-3$ is the kurtosis.

\subsubsection*{Generalizing the Fourier and Walsh syntheses}
To generalize the Fourier synthesis
we write
$h_j(x) = \sum_{k\in\ints}\wh h_j(k)e^{2\pi ikx}$
(in mean square) for
$\wh h_j(k) = \int_0^1 h_j(x)e^{-2\pi ikx}\rd x.$
We note that $\mu = \prod_j\mu_j$ where
$\mu_j = \wh h_j(0)$.
Now
$\wh f(\bsk) = \prod_{j=1}^d \wh h_j(k_j)$
and for even $p\ge 2$
\begin{align*}
\ult_u^{[p]}+\mu^p
& =
\prod_{j\not\in u}|\mu_j|^p\sum_{\bsk_u\in\ints^{|u|}}
\prod_{j\in u}|\wh h_j(k_j)|^p
 =
\prod_{j\not\in u}|\mu_j|^p
\prod_{j\in u}\Biggl(
\sum_{k_j\in\ints}|\wh h_j(k_j)|^p
\Biggr).
\end{align*}
Using the alternating sum~\eqref{eq:newsigsqrb}
and simplifying, we obtain
\begin{align*}
\sigma_u^{[p]}
& =
\prod_{j\not\in u}|\mu_j|^p
\prod_{j\in u}
\biggl(\,
\sum_{k_j\in\ints_*}|\wh h_j(k_j)|^p
\biggr)
\end{align*}
for $u\ne\emptyset$. The effect is to change $\ints$ to $\ints_*$
in the sums.

Given two functions $h_j$ with the same variance,
the measure
$\sum_{k_j\in\ints_*}|\wh h_j(k_j)|^p$, for $p>2$, is a measure
of sparsity for the spectrum.
It does not favor either
high or low frequencies. To put more emphasis on high or low frequencies one could use weighted coefficients as outlined in subsection~\ref{subsec_weighted_coeff}.

Analogous formulae hold for the Walsh synthesis.
Now we write the factors of $f$ as
$h_j(x) = \sum_{k \in \mathbb{N}_0} \wh h_{j, \wal}(k) \wal_k(x)$
for
$\wh h_{j, \wal}(k) = \int_0^1 h_j(x) \overline{\wal_k}(x) \rd x$.
Here $\mu_{\wal} = \prod_j \mu_{j, \wal}$ where $\mu_{j, \wal} = \wh h_{j,\wal}(0)$ and
$\wh f_{\wal}(\bsk) = \prod_{j=1}^d \wh h_{j,\wal}(k_j).$
For even $p \ge 2$
the same argument that we used in the Fourier case leads to
\begin{align*}
\sigma_{u, \wal}^{[p]}
& =
\prod_{j\not\in u}|\mu_{j, \wal}|^p
\prod_{j\in u}
\Biggl(
\sum_{k_j\in\natu}|\wh h_{j, \wal}(k_j)|^p
\Biggr).
\end{align*}
for $u\ne\emptyset$.

\subsubsection*{Generalizing the Sobol' identity}
When we generalize the Sobol' identity
we get
\begin{align*}
\ult^{(p)}_u+\mu^p &= \int
\prod_{k=1}^pf(\bsx_u\glu\bsz_{-u}^{(k)})
\mrd\bsx\prod_{k=1}^p\mrd\bsz^{(k)}
=
\prod_{j\in u}\int_0^1 h_j(x_j)^p\rd x_j
\prod_{j\not\in u}\mu_j^p.
\end{align*}
Where the Fourier synthesis had a $p$'th moment
$\sum_{k_j\in\ints}|\wh h_j(k_j)|^p$ of Fourier
coefficients, this approach has an ordinary
$p$'th moment $\int_0^1h_j(x)^p\rd x$.
Using the alternating sum~\eqref{eq:newsigsqra}
we obtain
\begin{align*}
\sigma^{(p)}_u
&=
\prod_{j\not\in u}\mu_j^p
\prod_{j\in u}\Bigl(\int_0^1h_j(x)^p\rd x-\mu_j^p\Bigr)
\end{align*}
for $u\ne\emptyset$.

For the generalized Sobol' identity we can make use
of the moments $\gamma_j$ and $\kappa_j$ of $h_j$.
The special cases of most interest have $p=3$ or $4$.
For $p=3$
$$\int_0^1h_j(x)^3\rd x = \mu_j^3+3\mu_j\tau_j^2+\gamma_j\tau_j^3$$
and so for $u\ne\emptyset$,
\begin{align*}
\ult_u^{(3)} & =
\prod_{j\not\in u}\mu_j^3
\prod_{j\in u}
\bigl(\mu_j^3+3\mu_j\tau_j^2+\gamma_j\tau_j^3\bigr)
-\mu^3
,\quad\text{and}\\
\sigma^{(3)}_u&=
\prod_{j\not\in u}\mu_j^3
\prod_{j\in u}\tau_j^2\bigl(3\mu_j+\gamma_j\tau_j\bigr).
\end{align*}
The $\sigma^{(3)}_u$ are
`components of skewness' analogues of the
components of variance $\sigma^2_u$.
Some of these components may be negative.
If every $\mu_j>0$ and every $\tau_j>0$,
then a negative component
of skewness arises if $3\mu_j+\gamma_j\tau_j<0$
holds for an odd number of indices $j\in u$.

Product functions illustrate one challenge with
importance measures taking negative values.
The same variable $x_j$ can drive the function
towards negative values through one component
$\sigma^{(3)}_u$ while driving it towards
positive values through another
component $\sigma^{(3)}_v$.
Similarly, whether the total effect $\ult^{(3)}_u$ is
positive or negative depends on the signs
of $\mu_j$ for $j\not\in u$.
These features make $p=3$ hard to interpret.

For $p=4$, we find
$$\int_0^1h_j(x)^4\rd x =
\mu_j^4+6\mu_j^2\tau_j^2+4\mu_j\gamma_j\tau_j^3+\kappa_j\tau_j^4
$$
and so for $u\ne\emptyset$,
\begin{align*}
\ult_u^{(4)} & =
\prod_{j\not\in u}\mu_j^4
\prod_{j\in u}
\bigl(
\mu_j^4+6\mu_j^2\tau_j^2+4\mu_j\gamma_j\tau_j^3+\kappa_j\tau_j^4
\bigr)
-\mu^4,\quad\text{and}\\
\sigma^{(4)}_u &=
\prod_{j\not\in u}\mu_j^4
\prod_{j\in u}\tau_j^2\bigl(6\mu_j^2+4\mu_j\gamma_j\tau_j+\kappa_j\tau_j^2\bigr).
\end{align*}
If $j\not\in u\ne\emptyset$ and $\mu_j\ne 0$, then
$$
\frac{\sigma^{(4)}_{u\cup\{j\}}}{\sigma^{(4)}_u}
=
\upsilon_j^2
\bigl(6+4\gamma_j\upsilon_j+\upsilon_j^2\kappa_j\bigr).
$$
where $\upsilon_j=\tau_j/\mu_j$ is the
$j$'th coefficient of variation.

A variable with a large absolute coefficient
of variation $|\upsilon_j|$ tends to raise
all of the $\sigma_u^{(4)}$ in which it participates
just as it does for the $p=2$ ANOVA case.
Additionally a variable with large fourth moment $\kappa_j$
becomes more important.  Variables with
large skewness $\gamma_j$ become more important
if $\gamma_j$ has the same sign as $\mu_j$
but less important if the opposite holds.  Both
of these findings are intuitively reasonable when
we are interested in driving $|f|$ to its largest
values.

\subsection{Indicators of rectangles}

A special case of the
product functions are indicator (characteristic)
functions of hyperrectangles.
These have $h_j(x) = 1$ for $x_j\in[x_{j*},x_{j*}+\epsilon_j)$
and $h_j(x)=0$ for $x \in [0,1) \setminus [x_{j*},x_{j*}+\epsilon_j)$,
so that $f(\bsx)$ is the indicator
of a hyperrectangle with volume $\prod_j\epsilon_j$.
For a binary function, all of the $x_j$ have
to be in their respective intervals for the
function to take the high value. This means that we
should expect important interactions. To model
a spiky function we would have
all of the $\epsilon_j$ be small. Then the
most important one should be the smallest one.
Here we let $\epsilon=\mu = \prod_{j=1}^d\epsilon_j$.

The generalization of Sobol's identity
works entirely with moments of $h_j$ and
so without loss of generality
$h_j(x) = 1$ for $x<\epsilon_j$ and is $0$
otherwise.
The generalization of the Walsh-based synthesis
is not invariant to the interval one chooses.
In this setting we prefer the Fourier-based synthesis.
Shifting the interval from $[0,\epsilon_j)$
to $[x_{*j},x_{*j}+\epsilon_j)$ for
 $0\le x_{0j}\le 1-\epsilon_j$
changes the phase but not the modulus of $\wh h_j(k)$
leaving the importance measures unchanged when
$p\ge 2$ is even.

For the generalized Sobol' index construction we find
for $u\ne\emptyset$
\begin{align*}
\ult_u^{(p)}
&= \prod_{j\not\in u}\epsilon_j^p \prod_{j\in u}\epsilon_j
-\epsilon^p
=\epsilon^p
\Bigl(\prod_{j\in u}\epsilon_j^{-(p-1)}-1\Bigr),\quad\text{and,}\\
\sigma_u^{(p)}
&= \prod_{j\not\in u}\epsilon_j^p\prod_{j\in u}
(\epsilon_j -\epsilon_j^p)
=\epsilon^p\prod_{j\in u}(\epsilon_j^{-(p-1)}-1).
\end{align*}
Variables with smaller $\epsilon_j$ are more
important than those with larger $\epsilon_j$
and the effect is magnified at larger $p$.
Both $\ult_u{(p)}$ and $\sigma_u^{(p)}$ are
always nonnegative for integers $p\ge 2$
without requiring $p$ to be even.

We now consider the Fourier synthesis for even $p\ge 2$.
After applying some trigonometric identities,
we find that the key quantity there,
replacing $\epsilon_j-\epsilon_j^p$ satisfies
$$
\sum_{k\in\ints_*}|\wh h_j(k)|^p
= 2\sum_{k=1}^\infty\Bigl( \frac{\sin(\pi k\epsilon_j)}{\pi k}\Bigr)^p
\equiv T_p(\epsilon_j).
$$
Thus
$
\sigma_u^{[p]} =\epsilon^p \prod_{j\in u}T_p(\epsilon_j)/\epsilon_j^p.
$
Lemma~\ref{lem:dimreduce} gives some insight into $T_p$ for $\epsilon_j<1/2$
as follows.
For $p=4$,
$\ult^{[4]}_u +\mu^4 = \prod_{j\in u}Q_4(\epsilon_j)$ where
for $f(x)=1_{x<\epsilon}$ and $0<\epsilon<1/2$,
\begin{align*}
Q_4(\epsilon)&= \int_{0}^1 \int_0^1 \int_0^1 f(y_0) f(y_1) f(y_2) f(\{y_0-y_1+y_2\}) \rd y_0 \rd y_1 \rd y_2 \\
&=   \int_{0}^\epsilon \int_0^\epsilon \int_0^\epsilon 1_{\{y_0-y_1+y_2\}< \epsilon} \rd y_0 \rd y_1 \rd y_2 \\
& =  \frac{2}{3} \epsilon^3.
\end{align*}
As a result we have the identity $T_4(\epsilon) = \frac23\epsilon^3-\epsilon^4$,
and so for $u\ne\emptyset$,
\begin{align*}
\ult_u^{[4]}
&=  \prod_{j \notin u} \epsilon_j^4 \prod_{j \in u} \frac{2}{3} \epsilon_j^3 - \epsilon^4 = \epsilon^4 \Bigl(\prod_{j\in u} \frac{2}{3} \epsilon_j^{-1}   - 1 \Bigr) ,\quad\text{and,}\\
\sigma_u^{[4]}
&= \epsilon^4 \prod_{j\in u} \Bigl(\frac{2}{3} \epsilon_j^{-1} - 1\Bigr).
\end{align*}
For even $p\ge 2$ we will find a quantity $Q_p(\epsilon)$
like $Q_4$ is a $p-1$ dimensional volume proportional to $\epsilon^{p-1}$.
As a result, the Fourier synthesis will use
importance factors which grow as $\epsilon_j^{-1}$
compared to $\epsilon_j^{-p+1}$ for the moment method.

\subsection{Additive functions}
It frequently happens that high dimensional functions
enountered in practice are very nearly additive.
For example \cite{cafmowen} find that a $360$
dimensional function motivated by a financial
valuation problem is very nearly an additive function
of its inputs.
It is desirable that a measure of variable importance
for additive functions should only give nonzero importance
to singletons $u=\{j\}$.

Here we consider additive functions
\begin{equation}\label{ex_sum}
f(\bsx) = \mu + \sum_{j=1}^dh_j(x_j)
\end{equation}
where $\int_0^1h_j(x)\rd x=0$,
$\int_0^1h_j(x)^2\rd x=\tau_j^2$,
$\int_0^1h_j(x)^3\rd x=\gamma_j$,
and
$\int_0^1h_j(x)^4\rd x=\kappa_j$.

For even integers $p\ge 2$ we find that
$\sigma^{[p]}_{\{j\}}=\sum_{k\ne 0}|\wh h_j(k)|^p$
and $\sigma^{[p]}_{\{j\},\wal}=\sum_{k\ne 0}|\wh h_j(k)|^p$
are the only nonzero components.

For integer $p\ge 2$,
\begin{align*}
\ult_u^{(p)}+\mu^p
&=
\int \prod_{k=1}^p\Bigl[
\mu + \sum_{j\in u}h_j(x_j)
+\sum_{j\not\in u}h_j(y_{j}^{(k)})
\Bigr]\rd\bsx\prod_{k=1}^p\rd\bsy^{(k)}\\
&=
\int
\Bigl[\mu + \sum_{j\in u}h_j(x_j)\Bigr]^p\rd\bsx.
\end{align*}
For $p=3$,
$\ult_u^{(3)}+\mu^3
=
\mu^3+3\mu\sum_{j\in u}\tau_j^2
+\sum_{j\in u}\gamma_j$,
so that
$\ult^{(3)}_u = \sum_{j\in u}(\mu\tau_j^2+\gamma_j)$.
Next
\begin{align*}
\sigma^{(3)}_u
&=
\sum_{v\subseteq u}(-1)^{|u-v|}\sum_{j\in v}(3\mu\tau_j^2+\gamma_j).
\end{align*}
Reversing the order of summation, we find that
$\sigma^{(3)}_u=0$ for $|u|>2$
and otherwise
$$
\sigma^{(3)}_{\{j\}} = 3\mu\tau_j^2+\gamma_j,
$$
compared to $\sigma^{(2)}_{\{j\}}=\tau_j^2$.
We see that the only nonzero components of skewness for an additive
function are for singletons.

The same simplification does not hold in general.
For $p=4$,
\begin{align*}
\ult^{(4)}_u +\mu^4& =
\mu^4 +
6\mu^2\sum_{j\in u}\tau^2_j
+4\mu\sum_{j\in u}\gamma_j
+\sum_{j\in u}\kappa_j
+\sum_{j\in u}\sum_{k\in u-\{j\}}\tau_j^2\tau_k^2,\quad\text{so,}\\
\ult^{(4)}_u& =
\sum_{j\in u}\bigl(
6\mu^2\tau^2_j+4\mu\gamma_j+\kappa_j-\tau^4_j
\bigr)
+\biggl(\,\sum_{j\in u}\tau_j^2\biggr)^2.
\end{align*}
As a result
\begin{align*}
\sigma^{(4)}_u
& =
\begin{cases}
6\mu^2\tau^2_j+4\mu\gamma_j+\kappa_j-\tau^4_j, & u=\{j\}\\
2\tau^2_j\tau_k^2, & u=\{j,k\},\ j\ne k\\
0, & |u|>2.
\end{cases}
\end{align*}



\section{Discussion}\label{sec:disc}

We have shown that it is possible to generalize the
ANOVA decomposition to higher order methods.
Working directly with either Sobol's identities
or with a synthesis of Fourier or Walsh terms
both lead to measures that can be estimated by
quadrature. For even values $p$ the generalizations
give non-negative importance measures. For odd
values of $p$ the Dirichlet kernel trick recovers
non-negative importance measures for the Fourier
and Walsh approaches. On test functions that we can
study analytically, we see that these measures can
identify variables which drive the
function towards its extreme values.

\section*{Acknowledgments}
We thank Gary Tang
for bringing the problem to our attention.
This work was supported in part by 
grant DMS-0906056 from the U.S.\ National Science Foundation. J. D. is supported in part by a Queen Elizabeth 2 Fellowship from the Australian Research Council.
\bibliographystyle{apalike}
\bibliography{sensitivity}

\end{document}

%% file: newcommandsnotbook.tex
\allowdisplaybreaks

\newcommand{\cd}{{\cal D}}

\newcommand{\bsa}{{\boldsymbol{a}}}

\newcommand{\bsk}{{\boldsymbol{k}}}
\newcommand{\bsu}{{\boldsymbol{u}}}

\newcommand{\bsx}{{\boldsymbol{x}}}
\newcommand{\bsy}{{\boldsymbol{y}}}
\newcommand{\bsz}{{\boldsymbol{z}}}

\newcommand{\bszero}{{\boldsymbol{0}}}

\newcommand{\natu}{\mathbb{N}}
\newcommand{\real}{\mathbb{R}}

\newcommand{\ints}{\mathbb{Z}}

\newcommand{\glu}{{\!\::\:\!}}

\newcommand{\tmod}{\ \mathsf{mod}\ }

\newcommand{\phm}{\phantom{-}}

\newcommand{\mrd}{\mathrm{d}}
\newcommand{\rd}{\mathrm{\, d}}

\newcommand{\wh}{\widehat}

\newcommand{\wal}{\mathrm{wal}}

\DeclareMathOperator*{\argmin}{argmin}

\renewcommand{\emptyset}{\varnothing}
\renewcommand{\ge}{\geqslant}
\renewcommand{\le}{\leqslant}

\newcommand{\dustd}{\mathbf{U}} 

\newcommand{\e}{{\mathbb{E}}} 

